\documentclass[12pt,dvips]{article}

\usepackage{amsmath,amsfonts,amssymb,amsthm,graphicx,verbatim,subfig}
\usepackage[all]{xy}

\ifx\pdftexversion\undefined

\usepackage[a4paper,colorlinks,link
=black,filecolor=black,citecolor=black,urlcolor=black,pdfstartview=FitH]{hyperref}
\else

\usepackage[a4paper,colorlinks,linkcolor=black,filecolor=black,citecolor=black,urlcolor=black,pdfstartview=FitH]{hyperref}
\fi

\font\sixbb=msbm6
\font\eightbb=msbm8
\font\twelvebb=msbm10 scaled 1095
\newfam\bbfam
\textfont\bbfam=\twelvebb \scriptfont\bbfam=\eightbb
                           \scriptscriptfont\bbfam=\sixbb
\def\bb{\fam\bbfam\twelvebb}
\newcommand{\Rea}{{\bb R}}

\newcommand{\FF}{{\bb F}}

\newtheorem{theorem}{\bf Theorem}[section]
\newtheorem{claim}[theorem]{\bf Claim}

\newtheorem{proposition}[theorem]{\bf Proposition}

\newtheorem{lemma}{Lemma}
\newcommand{\enp}{\begin{flushright} $\Box$ \end{flushright}}
\newcommand{\beq}[0]{\begin{equation}}
\newcommand{\enq}[0]{\end{equation}}
\newcommand{\dn}{\Delta_{n-1}}

\newcommand{\cf}{{\cal F}}
\newcommand{\C}{{\mathcal C}}
\newcommand{\E}{\mathbb E}

\newcommand{\cs}{{\cal S}}

\newcommand{\lk}{{\rm lk}}

\newcommand{\prob}{{\rm Pr}}

\newcommand{\cg}{{\cal G}}
\newcommand{\cu}{{\cal U}}

\newcommand{\tdkl}{{\cal T}_d(k,\lambda)}
\newcommand{\td}{{\cal T}_d}
\newcommand{\rinf}{R_{\infty}(Y)}

\newcommand{\namedref}[2]{\hyperref[#2]{#1~\ref*{#2}}}

\title{Collapsibility and vanishing of top homology\\ in random simplicial complexes}
\begin{document}
\author{L. Aronshtam\thanks{Department of Computer Science, Hebrew University, Jerusalem 91904,
    Israel. e-mail: liobsar@gmail.com~.} \and  N. Linial\thanks{Department of Computer Science, Hebrew University, Jerusalem 91904,
    Israel. e-mail: nati@cs.huji.ac.il~. Supported by ISF and BSF grants.}
  \and T. {\L}uczak\thanks{Faculty of Mathematics and Computer Science, Adam Mickiewicz
University, Pozna\'n, Poland. e-mail: tomasz@amu.edu.pl~. Supported by the Foundation for
Polish Science.}
    \and R. Meshulam\thanks{Department of Mathematics,
Technion, Haifa 32000, Israel. e-mail:
meshulam@math.technion.ac.il~. Supported by ISF and BSF grants. }
}

\maketitle
\pagestyle{plain}
\begin{abstract}
Let $\dn$ denote the $(n-1)$-dimensional simplex. Let $Y$ be a
random $d$-dimensional subcomplex of $\dn$ obtained by starting
with the full $(d-1)$-dimensional skeleton of $\dn$ and then adding
each $d$-simplex independently with probability $p=\frac{c}{n}$.
We compute an explicit constant $\gamma_d=\Theta(\log d)$
so that for $c < \gamma_d$ such a random simplicial complex either
collapses to a $(d-1)$-dimensional subcomplex or it contains $\partial \Delta_{d+1}$, the
boundary of a $(d+1)$-dimensional simplex. We conjecture this bound to be sharp.
\\
In addition we show that there exists a constant
$\gamma_d< c_d <d+1$ such that for any $c>c_d$ and a fixed field $\FF$, asymptotically almost surely $H_d(Y;\FF) \neq 0$.
\end{abstract}

\section{Introduction}

Let $G(n,p)$ denote the probability space of graphs on the vertex
set $[n]=\{1,\ldots,n\}$ with independent edge probabilities $p$. It is well known (see e.g. \cite{AS00})
that if $c>1$ then a graph $G \in G(n,\frac{c}{n})$  a.a.s. contains a cycle, while for $c<1$
\begin{equation}
\label{gconn}
\lim_{n \rightarrow \infty} \prob~[~G \in G(n,\frac{c}{n})~:~G~acyclic~] \geq
\sqrt{1-c}\cdot \exp(\frac{c}{2}+\frac{c^2}{4}).
\end{equation}
In this paper we consider the analogous question for $d$-dimensional random complexes.
There are two natural extensions to the notion of an acyclic graph. Namely, the
vanishing of the $d$-th homology, and collapsibility to a $(d-1)$-dimensional subcomplex.
These are the two questions we consider here. We provide an upper bound on the threshold
for the vanishing of the $d$-th homology and a lower bound (which we believe to be tight)
for the threshold for collapsibility.

For a simplicial complex $Y$, let $Y^{(i)}$ denote the $i$-dimensional skeleton of $Y$. Let $Y(i)$
be the set of $i$-dimensional simplices of $Y$ and let $f_i(Y)=|Y(i)|$.
Let $\dn$ denote the $(n-1)$-dimensional simplex on the vertex set
$V=[n]$. For $d \geq 2$ let $Y_d(n,p)$ denote the probability space of complexes
$\Delta_{n-1}^{(d-1)} \subset Y \subset \Delta_{n-1}^{(d)}$ with
probability measure
$$\Pr(Y)=p^{f_d(Y)}(1-p)^{\binom{n}{d+1}-f_d(Y)}~.$$
Let $\FF$ be an arbitrary fixed field and let
$H_i(Y)=H_i(Y;\FF)$ and $H^i(Y)=H^i(Y;\FF)$ denote the $i$-th homology and cohomology groups of $Y$ with coefficients in $\FF$. Let $h_i(Y)=\dim_{\FF} H_i(Y)$ and $h^i(Y)=\dim_{\FF} H^i(Y)$.
Kozlov \cite{K10} proved the following
\begin{theorem}[Kozlov]
\label{kozlov}
For any function $\omega(n)$ that tends to infinity
$$\lim_{n \rightarrow \infty} \prob ~[~Y \in Y_d(n,p):
H_{d}(Y) \neq 0 ~]= \left\{
\begin{array}{ll}
        1 & p=\frac{\omega(n)}{n} \\
        0 & p=\frac{1}{\omega(n)n}.
\end{array}
\right.~~
$$
\end{theorem}
Our first result improves the upper bound in Theorem \ref{kozlov}.
Let $$g_d(x)=(d+1)(x+1)e^{-x}+x(1-e^{-x})^{d+1}~~$$
and let $c_d$ denote the unique positive solution of the equation $g_d(x)=d+1$.
A direct calculation yields that $c_d=d+1 - \Theta(\frac{d}{e^d})$.
\begin{theorem}
\label{upperbb}
For a fixed $c>c_d$
\begin{equation}
\label{genp}\lim_{n \rightarrow \infty} \prob ~[Y \in Y_d(n,\frac{c}{n}):~H_d(Y) \neq 0 ]= 1~.
\end{equation}
\end{theorem}
\noindent
{\bf Remark:} In the $2$-dimensional case, Theorem \ref{upperbb} implies that if $c>c_2 \simeq 2.783$ then
$Y \in Y_2(n,\frac{c}{n})$ a.a.s. satisfies $H_2(Y)\neq 0$. Simulations indicate that the actual threshold is somewhat lower (around $2.75$).

We next turn to collapsibility.
A $(d-1)-$dimensional simplex $\tau \in \dn (d-1)$ is a {\it free face} of a complex $Y \subset \dn^{(d)}$
if it is contained in a unique $\sigma \in Y(d)$. Let $R(Y)$ denote the complex obtained by removing
all free $(d-1)$-faces of $Y$ together with the $d$-simplices that contain them. We say that $R(Y)$ is obtained from $Y$ by a
$d$-collapse step. Let $R_0(Y)=Y$ and for $i \geq 1$ let $R_i(Y)=R(R_{i-1}(Y))$. We say that
$Y$ is {\it $d$-collapsible} if $\dim R_{\infty}(Y)<d$.
Costa, Farber and Kappeler \cite{CFK10} proved the following
\begin{theorem}[Costa, Farber and Kappeler]
\label{cfk}
If $\omega(n) \rightarrow \infty$
then $Y \in Y_2(n,\frac{1}{\omega(n)n})$ is a.a.s. $2$-collapsible.
\end{theorem}

Our second result improves Theorem \ref{cfk} and the lower bounds in Theorem~\ref{kozlov} as
follows.
Let $$u_d(\gamma,x)=\exp(-\gamma (1-x)^d)-x~~.$$
For small positive $\gamma$, the only solution of $u_d(\gamma,x)=0$ is $x=1$.
Let $\gamma_d$ be the infimum of the set of all nonnegative $\gamma$'s for which the equation $u_d(\gamma,x)=0$ has a solution $x<1$.
More explicitly,
$\gamma_d=(d x(1-x)^{d-1})^{-1}$ where $x$ satisfies $\exp(-\frac{1-x}{dx})=x$.
It is not hard to verify that this yields
\[
\gamma_d = \log d + O(\log \log d).
\]
Let $\cf_{n,d}$ denote the family of all $\dn^{(d-1)} \subset Y \subset \dn^{(d)}$ that do not contain the boundary of a $(d+1)-$simplex.
\begin{theorem}
\label{collapse}
Let $c<\gamma_d$ be fixed. Then in the probability space $Y_d(n,\frac{c}{n})$
\begin{equation}
\label{genp}\lim_{n \rightarrow \infty} \prob ~[Y~is~d-collapsible~|~Y \in  \cf_{n,d}]= 1~.
\end{equation}
\end{theorem}
\noindent
{\bf Remark:}
We have calculated $\gamma_2 \simeq 2.455$, and computer simulations suggest that this is indeed the actual threshold for collapsibility
for random complexes in $\cf_{n,2}$. Also,
$\gamma_3 \simeq 3.089, \gamma_4 \simeq 3.508$ and $\gamma_{100} \simeq 7.555.$
\ \\ \\
Clearly, if $Y$ is $d$-collapsible then $Y$ is homotopy equivalent to a $(d-1)$-dimensional complex, and in particular $H_d(Y)=0$. Hence

$$\lim_{n \rightarrow \infty} \prob ~[H_d(Y)=0~|~Y \in  \cf_{n,d}]= 1~$$
and
$$\lim_{n \rightarrow \infty} \prob ~[H_d(Y)=0]=\lim_{n \rightarrow \infty} \prob [\cf_{n,d}]=
\exp(-\frac{c^{d+2}}{(d+2)!})~.$$

The paper is organized as follows. In Section \ref{s:upper} we prove Theorem \ref{upperbb}.
In Section \ref{s:tree} we analyze a random $d$-tree process that underlies
our proof that for  $c< \gamma_d$, a random
$Y \in Y_d(n,\frac{c}{n}) \cap \cf_{n,d}$ is a.a.s. $d$-collapsible. Another main ingredient of the proof is
an upper bound on the number of minimal non $d$-collapsible complexes given in Section \ref{s:noncoll}.
In Section \ref{s:thold} we combine these results to derive Theorem \ref{collapse}.
We conclude in Section \ref{s:remark} with some comments and open problems.

\section{The Upper Bound}
\label{s:upper}
\ \ \ \
Let $Y \in Y_d(n,p)$. Then $h_i(Y)=0$ for $0<i<d-1$ and $f_i(Y)=\binom{n}{i+1}$ for $0 \leq i \leq d-1$.
The Euler-Poincar\'{e} relation $$\sum_{i \geq 0} (-1)^i
f_i(Y)=\sum_{i \geq 0} (-1)^i h_i(Y)$$
therefore implies
\begin{equation}
\label{ep}
h_d(Y)=f_d(Y)-\binom{n-1}{d}+h_{d-1}(Y).
\end{equation}
The inequality $h_d(Y) \geq f_d(Y)- \binom{n-1}{d}$ already implies that if $c> d+1$ then a.a.s. $h_d(Y) \neq 0$.
This was also observed in the $2$-dimensional case by Costa, Farber and Kappeler  \cite{CFK10}.
The idea of the proof of Theorem \ref{upperbb} is to improve this estimate by providing a
non-trivial lower bound on $E[h_{d-1}]$.
For $\tau \in \dn(d-1)$ let $$\deg_Y(\tau)=|\{\sigma \in Y(d): \tau \subset \sigma\}|$$ and let
$$A_{\tau}=\{Y \in Y_d(n,p): \deg_Y(\tau)=0\}~~.$$
For $\sigma \in Y(d)$ let $L_{\sigma}$ be the subcomplex of $\sigma^{(d-1)}$ given by
$$L_{\sigma} =\sigma^{(d-2)} \cup \{\tau \in \sigma^{(d-1)}:\deg_Y(\tau)>1\}~.$$
Let $P_{n,d}$ denote the family of all pairs $(\sigma,L)$, such that $\sigma \in \dn(d)$ and
$\sigma^{(d-2)} \subset L \subset \sigma^{(d-1)}$. For $(\sigma,L) \in P_{n,d}$ let
$$B_{\sigma,L}= \{Y \in Y_d(n,p): \sigma \in Y~,~L_{\sigma} =L \}~~.$$
The space of $i$-cocycles of a complex $K$ is as usual denoted by $Z^i(K)$.
The space of relative $i$-cocycles of a pair $K' \subset K$ is denoted by $Z^i(K,K')$
and will be identified with the subspace of $i$-cocycles of $K$ that vanish on $K'$.
Let $z^i(K)=\dim Z^i(K)$ and $z^i(K,K')=\dim Z^i(K,K')$. For a $(d-1)$-simplex $\tau=[v_0,\ldots,v_{d-1}]$, let
$1_{\tau}$ be the indicator $(d-1)$-cochain of $\tau$ (i.e. $1_{\tau}(\eta)={\rm sgn}(\pi)$ if
$\eta=[v_{\pi(0)},\ldots,v_{\pi(d-1)}]$  and is zero otherwise).
If $\tau \in \dn(d-1)$ then $Z^{d-1}(\tau)$ is the $1$-dimensional space spanned by $1_{\tau}$.
If $(\sigma,L) \in P_{n,d}$ and $f_{d-1}(L)=j$, then  $z^{d-1}(\sigma,L)=d-j$.
indeed, suppose $\sigma=[v_0,\ldots,v_d]$ and for $0 \leq i \leq d$ let $\tau_i=[v_0,\ldots,\widehat{v_i},\ldots,v_d]$. If $L(d-1)=\{\tau_i\}_{i=d-j+1}^d$
then  $\{1_{\tau_0}-1_{\tau_i}\}_{i=1}^{d-j}$ forms a basis of $Z^{d-1}(\sigma,L)$ .

\begin{claim}
\label{cocycles}
For  any $Y \in Y_d(n,p)$
$$
Z^{d-1}(Y) \supset \bigoplus_{\{\tau \in \dn(d-1)~:~Y \in A_{\tau}\}} Z^{d-1}(\tau) \oplus
\bigoplus_{\{(\sigma,L) \in P_{n,d}~:~Y \in B_{\sigma,L}\}} Z^{d-1}(\sigma,L)~~.
$$
\end{claim}
\noindent
{\bf Proof:} The containment is clear. To show that the right hand side is a direct sum note that
nontrivial cocycles in different summands must have disjoint supports and are therefore linearly independent.
{\enp}
\noindent
Let $$a(Y)=|\{\tau \in \dn(d-1)~:~\deg_Y(\tau)=0\}|$$ and for $0 \leq j \leq d$ let
$$\alpha_j(Y)=|\{(\sigma,L) \in P_{n,d}~:~Y \in B_{\sigma,L}~,~f_{d-1}(L)=j\}|.$$
Note that  $\alpha_j(Y)$ is the number of $d$-faces of $Y$ that contain exactly $d+1-j$ $(d-1)$-faces of degree 1.
By Claim \ref{cocycles}
\begin{equation}
\label{lbz}
z^{d-1}(Y) \geq u(Y)\stackrel{\text{def}}{=}a(Y)+\sum_{j=0}^{d} \alpha_j(Y) (d-j)~~.
\end{equation}
As $h_{d-1}(Y)=h^{d-1}(Y)= z^{d-1}(Y)-\binom{n-1}{d-1}$, it follows from (\ref{ep}) and (\ref{lbz}) that
\begin{equation}
\label{lbh}
h_d(Y) \geq v(Y)\stackrel{\text{def}}{=}f_d(Y)+u(Y)-\binom{n}{d}~~.
\end{equation}
Theorem \ref{upperbb} will thus follow from
\begin{theorem}
\label{larged}
Let $c> c_d$ and let $p=\frac{c}{n}$. Then
$$\lim_{n \rightarrow \infty} \prob ~[~Y \in Y_d(n,p):v(Y) \leq 0]=0~.$$
\end{theorem}
\noindent
{\bf Proof:} First note that
$$E[f_d]=\binom{n}{d+1}p=\frac{c}{(d+1)!} n^{d}-O(n^{d-1})~,$$
$$E[a]=\binom{n}{d}(1-p)^{n-d}=\frac{e^{-c}}{d!}n^{d}-O(n^{d-1})~,$$
and for $0 \leq j \leq d$
$$
E[\alpha_j]=\binom{n}{d+1} \binom{d+1}{j} p (1-p)^{(n-d-1)(d+1-j)} (1-(1-p)^{n-d-1})^{j}=$$
$$
\frac{n^d c }{(d+1)!}\binom{d+1}{j}e^{-c(d+1-j)}(1-e^{-c})^j-O(n^{d-1})~.$$
Therefore
$$E[u]=E[a]+\sum_{j=0}^{d} E[\alpha_j](d-j)=$$
$$\frac{n^d e^{-c}}{d!}+\frac{n^d c}{(d+1)!}  \sum_{j=0}^d \binom{d+1}{j}e^{-c(d+1-j)}(1-e^{-c})^j (d-j)-O(n^{d-1}) =$$
$$\frac{n^d}{(d+1)!}((1+c)(d+1)e^{-c}-c(1-(1-e^{-c})^{d+1}))-O(n^{d-1}).$$
It follows that
$$E[v]=E[f_d]+E[u]-\binom{n}{d}=$$ $$\frac{n^d}{(d+1)!} (c+(1+c)(d+1)e^{-c}-c(1-(1-e^{-c})^{d+1})-(d+1))-O(n^{d-1})=$$
$$\frac{n^d}{(d+1)!} (g_d(c)-d-1) -O(n^{d-1})~.$$
Since $c>c_d$ it follows that for sufficiently large $n$
\begin{equation}
E[v] \geq \epsilon n^d
\end{equation}
where $\epsilon>0$ depends only on $c$ and $d$.
To show that $v$ is a.a.s. positive we
use the following consequence of Azuma's inequality due to McDiarmid \cite{McDiarmid}.
\begin{theorem}
\label{azuma}
Suppose $f:\{0,1\}^m \rightarrow \Rea$ satisfies
$|f(x)-f(x')| \leq T$ if $x$ and $x'$ differ in at most one coordinate.
Let $\xi_1,\ldots,\xi_m$ be independent $0,1$ valued random variables
and let $F=f(\xi_1,\ldots,\xi_m)$. Then for all $\lambda>0$
\begin{equation}
\label{azumain}
\prob[F \leq E[F] - \lambda] \leq \exp(-\frac{2\lambda^2}{T^2 m})~.
\end{equation}
\end{theorem}

Let $m=\binom{n}{d+1}$ and let $\sigma_1,\ldots,\sigma_m$ be an arbitrary ordering
of the $d$-simplices of $\dn$. Identify $Y \in Y_d(n,p)$ with its indicator vector
$(\xi_1,\ldots,\xi_m)$ where $\xi_i(Y)=1$ if $\sigma_i \in Y$ and $\xi_i(Y)=0$ otherwise.
Note that if $Y$ and $Y'$ differ in at most one $d$-simplex then
$|a(Y)-a(Y')|\leq d+1$ and $|\alpha_j(Y)-\alpha_j(Y')| \leq d+1$ for all $0 \leq j \leq d$.
It follows that
$|v(Y)-v(Y')| \leq T=2d^3$. Applying McDiarmid's inequality (\ref{azumain}) with $F=v$
and $\lambda = E[v]$ it follows that
$$
\prob[v \leq 0] \leq \exp(-\frac{2 E[v]^2}{(2d^3)^2 m})\leq \exp(-C_2 n^{d-1})
$$
for some $C_2=C_2(c,d)>0$. {\enp}
\ \\ \\
{\bf Remark:} The approach used in the proof of Theorem \ref{upperbb} can be extended as follows.
For a fixed $\ell$, let $Z_{(\ell)}^{d-1}(Y) \subset Z^{d-1}(Y)$ denote the subspace spanned by
$(d-1)$-cocycles $\phi \in Z^{d-1}(Y)$ such that $|supp(\phi)| \leq \ell$. Let
$$\theta_{d,\ell}(x)=\lim_{n \rightarrow \infty} \frac{E[\dim  Z_{(\ell)}^{d-1}(Y)]}{\binom{n}{d}}$$
where the expectation is taken in the probability space $Y_d(n,\frac{x}{n})$.
For example, it was shown in the proof of Theorem \ref{upperbb} that
$\theta_{d,1}(x)=e^{-x}$ and $$\theta_{d,2}(x)=(1+x)e^{-x}-\frac{x}{d+1}(1-(1-e^{-x})^{d+1}).$$
Let $x=c_{d,\ell}$ denote the unique positive root of the equation
$$x+(d+1)\theta_{d,\ell}(x)=d+1.$$
The following fact is implicit in the proof of Theorem \ref{upperbb}.
\begin{proposition}
\label{upperbbb}
For any fixed $c>c_{d,\ell}$
\begin{equation}
\label{genpp}\lim_{n \rightarrow \infty} \prob ~[Y \in Y_d(n,\frac{c}{n}):~H_d(Y) \neq 0 ]= 1~.
\end{equation}
\end{proposition}
\noindent
Let $\tilde{c}_d=\lim_{\ell \rightarrow \infty} c_{d,\ell}$.
It seems likely that $\frac{\tilde{c}_d}{n}$
 is the exact threshold for the vanishing
of $H^d(Y)$. This is indeed true in the graphical case $d=1$.
\begin{proposition}
\label{grph}
$$\tilde{c}_1=1.$$
\end{proposition}
\noindent
{\bf Proof:} For a subtree $K=(V_K,E_K)$ on the vertex set $V_K \subset [n]$ let
$A_K$ denote all graphs $G \in G(n,p)$ that contain $K$ as an induced subgraph and contain no edges in the cut $(V_K,\overline{V_K})$. The space of $0$-cocycles $Z^0(K)$ is $1$-dimensional and is spanned by the indicator function of $V_K$.
As in Claim \ref{cocycles} it is clear that for
$G \in G(n,p)$ and a fixed $\ell$
$$
Z^{0}(G) \supset \bigoplus_{\{K: |V_K| \leq \ell~and~G \in A_K\}}  Z^{0}(K).
$$
Hence for $p=\frac{x}{n}$
$$E[\dim Z_{(\ell)}^0(G)] =\sum \{\prob[A_K]~:~K~is~a~tree~on~ \leq \ell ~vertices\}=$$
$$
\sum_{k=1}^{\ell} \binom{n}{k} k^{k-2}(\frac{x}{n})^{k-1} (1-\frac{x}{n})^{k(n-k)+\binom{k-1}{2}} \sim$$
$$n \sum_{k=1}^{\ell}\frac{k^{k-2}}{k!} x^{k-1} e^{-xk}=\frac{n}{x}
\sum_{k=1}^{\ell}\frac{k^{k-2}}{k!} (x e^{-x})^k .$$
Let $T(z)=\sum_{k=1}^{\infty} \frac{k^{k-2}}{k!} z^k$ be the exponential generating function for the number of trees. Then
$$\lim_{\ell \rightarrow \infty} \theta_{1,\ell}(x)=\lim_{\ell,n \rightarrow \infty} \frac{E[\dim Z_{(\ell)}^0(G)]}{n}=
\frac{T(x e^{-x})}{x}.$$
Therefore $\tilde{c}_1=\lim_{\ell \rightarrow \infty} c_{1,\ell}$ is the solution
of the equation
\begin{equation}
\label{teq}
x+ \frac{2T(x e^{-x})}{x}=2.
\end{equation}
Let $R(z)=\sum_{k=1}^{\infty} \frac{k^{k-1}}{k!} z^k$
be the exponential generating function for the number of rooted trees.
It is classically known (see e.g. \cite{Stanley}) that $R(z)=z\exp(R(z))$, and that $T(z)=R(z)-\frac{1}{2}R(z)^2$.
It follows that $R(e^{-1})=1$ and $T(e^{-1})=\frac{1}{2}$. Hence
$\tilde{c}_1=1$ is the unique solution of (\ref{teq}).
{\enp}

\section{The Random $d$-Tree Process}
\label{s:tree}

A simplicial complex $T$ on the vertex set $V$ with $|V|=\ell \geq d$ is a {\it $d$-tree}
if there exists an ordering  $V=\{v_1,\ldots,v_{\ell}\}$ such that $\lk(T[v_1,\ldots,v_{i}],v_i)$
is a $(d-1)$-dimensional simplex for all $d+1 \leq i \leq \ell$.
Let $G_T$ denote the graph
with vertex set $T(d-1)$, whose edges are the pairs $\{\tau_1,\tau_2\}$
such that $\tau_1 \cup \tau_2 \in T(d)$. Let $dist_T(\tau_1,\tau_2)$ denote the distance between $\tau_1$ and $\tau_2$ in the graph $G_T$.

A {\it rooted $d$-tree} is a pair $(T,\tau)$ where $T$ is a $d$-tree and $\tau$ is some $(d-1)$-face of $T$. Let $\tau$ be a fixed $(d-1)$-simplex. Given  $k \geq 0$ and $\gamma >0$ we describe a random process that gives rise to a probability space $\tdkl$ of all $d$-trees $T$ rooted at $\tau$ such that $dist_T(\tau,\tau') \leq k$ for all $\tau' \in T$.
The definition of $\tdkl$ proceeds by induction on $k$. $\td(0,\gamma)$ is the $(d-1)$-simplex $\tau$.
Let $k \geq 0$. A $d$-tree in $\td(k+1,\gamma)$ is generated as follows: First generate a $T \in \tdkl$ and let $\cu$ denote all $\tau' \in T(d-1)$  such that $dist_T(\tau,\tau')=k$.
Then, independently for each $\tau' \in \cu$, pick $J=J_{\tau'}$ new vertices $z_1,\ldots,z_J$ where $J$ is Poisson distributed with parameter $\gamma$, and add the $d$-simplices $z_1\tau', \ldots,z_J\tau'$ to $T$.

We next define the operation of pruning of a rooted $d$-tree
$(T,\tau)$. Let $\{\tau_1,\ldots,\tau_{\ell}\}$ be the set of all free $(d-1)$-faces of $T$ that are distinct from $\tau$, and let $\sigma_i$ be the unique $d$-simplex of $T$ that contains $\tau_i$. The $d$-tree $T'$ obtained from $T$ by removing the simplices $\tau_1,\sigma_1,\ldots,\tau_{\ell},\sigma_{\ell}$ is called the {\it pruning of} $T$.
Clearly, any $T \in \td(k+1,\gamma)$ collapses to its root $\tau$  after at most $k+1$ pruning steps.
Denote by  ${\cal C}_d(k+1,\gamma)$ the event that $T \in \td(k+1,\gamma)$ collapses to $\tau$ after at most $k$ pruning steps, and let $\rho_d(k,\gamma)=\prob[{\cal C}_d(k+1,\gamma)]$. Clearly, $\rho_d(0,\gamma)$ is the probability that $T \in \td(1,\gamma)$ consists only of $\tau$, hence
\begin{equation}
\label{kzero}
\rho_d(0,\gamma)=e^{-\gamma}.
\end{equation}
Let $\sigma_1,\ldots,\sigma_j$ denote the $d$-simplices of $T \in \td(k+1,\gamma)$ that contain $\tau$
and for each $1 \leq i \leq j$ let $\eta_{i1},\ldots,\eta_{id}$ be the $(d-1)$-faces of $\sigma_i$ that are different from $\tau$. Let $T_{i \ell} \in \tdkl$ denote the subtree of $T$ that grows out of
$\eta_{i \ell}$. Clearly, $T$ collapses to $\tau$ after at most $k$ pruning steps iff for each $1 \leq i \leq j$, at least one of the $d$-trees $T_{i \ell}$ collapses to its root $\eta_{i \ell}$ in at most $k-1$ steps.
We therefore obtain the following recursion:
$$
\rho_d(k,\gamma)=\sum_{j=0}^{\infty}\prob~[J=j] (1-(1-\rho_d(k-1,\gamma))^d)^j=
$$
$$
\sum_{j=0}^{\infty}\frac{\gamma^j}{j!} e^{-\gamma} (1-(1-\rho_d(k-1,\gamma))^d)^j=
$$
\begin{equation}
\label{recur}
\exp(-\gamma(1-\rho_d(k-1,\gamma))^d)~.
\end{equation}
Equations (\ref{kzero}) and (\ref{recur}) imply that the sequence $\{\rho_d(k,\gamma) \}_k$
is non-decreasing and converges to $\rho_d(\gamma) \in (0,1]$, where $\rho_d(\gamma)$ is the smallest positive solution of the equation
\begin{equation}
\label{eqcoll}
u_d(\gamma,x)=\exp(-\gamma (1-x)^d)-x=0.
\end{equation}
If $\gamma \geq 0$  is small, then  $\rho_d(\gamma)=1$.
Let $\gamma_d$ denote the infimum of the set of nonnegative $\gamma$'s for which
$\rho_d(\gamma)<1$. The pair $(\gamma,x)=(\gamma_d,\rho_d(\gamma_d))$ satisfies
both $u_d(\gamma_d,\rho_d(\gamma_d))=0$ and
$\frac{\partial u_d}{\partial x}(\gamma_d,\rho_d(\gamma_d))=0$.
A straightforward computation shows that $\gamma_d=(d x(1-x)^{d-1})^{-1}$
where $x=\rho_d(\gamma_d)$ is the unique solution of $\exp(-\frac{1-x}{dx})=x$.

\section{The Number of Non-$d$-Collapsible Complexes}
\label{s:noncoll}

When we discuss $d$-collapsibility, we only care about the inclusion relation
between $d$-faces and $(d-1)$-faces. Therefore, in the present section we can and will simplify matters
and consider only the complex that is induced from our (random) choice of $d$-faces.
Namely, for every $i \le d$, a given $i$-dimensional face belongs to the complex iff
it is contained in some of the chosen $d$-faces.

A complex is a {\em core} if every $(d-1)$-dimensional face belongs to at least
two simplices, so that {\em not even a single} collapse step is possible.

A core complex is called a {\em minimal core} complex if none of its proper subcomplexes is a core.

The main goal of this section is to show that with almost certainty
there are just two types of
minimal core subcomplexes that a sparse random complex can have. It can either be
the boundary of a $(d+1)$-simplex, $\partial\Delta_{d+1}$, or
it must be very large. Obviously this implies that there are no small non-collapsible subcomplexes which do not contain
the boundary of a $(d+1)$-simplex.

\begin{theorem}
\label{subcmplx}
For every $c>0$ there exists a constant $\delta=\delta(c)>0$ such that
a.a.s. every minimal core  subcomplex $K$ of $Y \in Y_d(n,\frac{c}{n})$
with $f_d(K) \le \delta n^d$, must contain the boundary of a $(d+1)$-simplex.
\end{theorem}

Henceforth we use the convention that {\em faces} refer to arbitrary
dimensions, but unless otherwise specified, the word {\em simplex} is reserved to
mean a $d$-face.

Our proof uses the first moment method. In the main step
of the proof we obtain an upper bound
on $\C_d(n,m)$, the number of all minimal core
$d$-dimensional complexes on vertex set $[n]=\{1,2,\dots,n\}$,
which contain $m$ simplices.

Two simplices are considered {\em adjacent}
if their intersection is a $(d-1)$- face.
If $A \stackrel{\cdot}{\cup} B$ is a splitting of a minimal core complex,
then there is a simplex in $A$ and one in $B$ that are adjacent,
otherwise the corresponding subcomplexes are cores as well.
Therefore $K$ can be constructed by successively adding a simplex that is adjacent to an already
existing simplex. This consideration easily yields an upper bound of $n^{d+m}$
on $\C_d(n,m)$. The point is that if $m= \delta n^d$ for $\delta > 0$
small enough, we get an exponentially smaller (in $m$) upper bound
and this is crucial for our analysis.

\begin{lemma}\label{l:1}
Let $m= \delta n^d$ and $\delta>0$ small enough. Then
\begin{equation}\label{e4}
\C_d(n,m)
\le {n^{d-1}\choose (d^{2}m)^{\frac{d-1}{d}}}n^d n^m (2^{d+1}d^3\delta^{\frac{1}{d^4}})^m .
\end{equation}
\end{lemma}

\begin{proof}

Let $b=\left(\frac{d(d+1)\delta}{2}\right)^{\frac{1}{d}}$.
A $(d-2)$-face is considered \emph{heavy} or \emph{light}
depending on whether it is covered by at least
$bn$ $(d-1)$-faces or less. The sets of heavy and light
$(d-2)$-faces are denoted by
$H_{d-2}$ and $L_{d-2}$ repectively. We claim that $|H_{d-2}|\le b^{d-1}n^{d-1}$.
To see this note that each simplex contains
exactly $d+1$ $(d-1)$-faces, but the complex is a core, so that each
$(d-1)$-face is covered at least twice. Consequently, our complex has
at most $\frac{m(d+1)}{2}$ $(d-1)$-faces.
Likewise, each $(d-1)$-face contains $d$ $(d-2)$-faces.
Each heavy $(d-2)$-face is covered at least $bn$ times and the claim follows by the
following calculation:

$$
|H_{d-2}|\le \frac{d(d+1)m}{2bn}=\frac{d(d+1)\delta n^d}{2bn}=\frac{d(d+1)\delta n^{d-1}}{2b}=b^{d-1}n^{d-1}.
$$

We extend the heavy/light dichotomy to lower dimensions as well.
For each $0\le i\le d-3$, an $i$-face is considered \emph{heavy}
if it covered by at least $b\cdot n$ \emph{heavy }$(i+1)$-faces.
Otherwise it is \emph{light}. The sets of heavy/light $i$-faces are denoted
by $H_{i}$ resp. $L_{i}$. By counting inclusion relations between heavy faces
of consecutive dimensions it is easily seen that $|H_i| \le \frac{(i+2)|H_{i+1}|}{bn}$
which yields
$$
|H_{i}|\le \frac{(d-1)!}{(i+1)!}b^{i+1}n^{i+1}.
$$

The set of $i$-dimensional heavy (resp. light) faces
contained in a given face $\sigma$ is denoted
by $H_i^{\sigma}$ be (resp. $L_i^{\sigma}$).

The bulk of the proof considers a sequence of complexes $C_1,\dots, C_m=C$, where the complex
$C_{i}$ is obtained from $C_{i-1}$ by adding a single simplex. A $(d-1)$-face $\sigma$ of $C_i$ can be {\em saturated} or {\em unsaturated}. This depends on whether or not
every simplex in $C_m$ that contains $\sigma$ already belongs to $C_i$. Prior to defining the complexes $C_i$,
we specify the set of heavy $(d-2)$-faces in one of
at most ${{n\choose {d-1}}\choose b^{d-1}n^{d-1}}$ possible ways. Note that this choice uniquely determines
the sets of heavy faces for every dimension $0\le i\le d-3$.
We start off with the complex $C_{1}$, which has exactly one
simplex. Clearly there are ${n\choose d+1}$ possible choices for $C_1$.
We move from $C_{i-1}$ to $C_{i}$ by adding a single simplex $t_{i}$,
which covers a \emph{chosen} unsaturated $(d-1)$-face $\sigma_{i-1}$ of $C_{i-1}$.
Our choices are subject to the condition
that every heavy $(d-2)$-face in $C_m$ is one of the heavy $(d-2)$-faces
chosen prior to the process. In other words, we must never make choices that create
any additional heavy faces in addition to those derived from our preliminary choice.
Our goal is to bound the number of choices for this process.

The crux of the argument is a rule for selecting the chosen face.
Associated with every face is a vector counting the number of its heavy vertices, its heavy edges,
its heavy 2-faces etc. The chosen face is always lexicographically minimal w.r.t. this vector,
breaking ties arbitrarily.
A $(d-1)$-face all of whose subfaces are light is called \emph{primary}.

In each step $j$ we expand a $(d-1)$-face $\sigma$ to a simplex $\sigma \cup y$.
Such a step is called a {\em saving} step if either:
\begin{enumerate}
\item The vertex $y$ is heavy.
\item There exists a light $(d-2)$- subface $\tau\subset\sigma$ such that $\tau\cup y$
is contained in a simplex in $C_{j-1}$.
\item There exists a light subface $\tau\subset \sigma$ such that the face $\tau\cup y$ is heavy.
\end{enumerate}

Note that the number
of choices of $y$ in the first case is $\le |H_0| \le (d-1)!\cdot b\cdot n$.
In the second case the number of choices for $y$ is at most $d bn$.
In the third case there are $d-2$ possibilities for the dimension of the light face and for each such dimension $i$
there are at most $ {d \choose i+1} bn$ choices for $y$.
In all cases the number of choices for $y$ is at most $\le d^{d}bn$.
A step that is not saving is considered {\em wasteful}. For wasteful steps we bound the number of
choices for $y$ by $n$.\\

The idea of the proof is that every
such a process which produces a minimal core complex
must include many saving steps.
More specifically, we want to show:

\begin {claim}
\label{d_cube}
For every $d^3$ wasteful steps, at least one saving step is carried out.
\end{claim}
\begin{proof}


The proof of this claim consists of two steps.
We show that there is no sequence of $d(d-1)$ consecutive wasteful steps, without the creation of an
unsaturated primary face.
Also, the creation of $d+1$ primary faces necessarily involves a saving step.

If $u$ is a vertex in a $(d-1)$-face $\sigma$, let $r^{\sigma}_i(u)$ be the number of heavy
$i$-faces in $\sigma$ that contain $u$. Also, $V_i^{\sigma}$ denotes the set of vertices $v$ in
$\sigma$ that are included only in light $i$-subfaces of $\sigma$.

\begin{proposition}
 \label{prop1}
Let $\sigma$ and $\sigma'$ be two consecutively chosen faces where $\sigma$ is non-primary and the extension step
on $\sigma$ is wasteful. Then $\sigma'$ precedes $\sigma$ in the order of faces and $|V_i^{\sigma'}|\ge |V_i^{\sigma}|+1$
where $i$ is the smallest dimension for which $|H_i^{\sigma}|>0$.
\end{proposition}
\begin{proof}

Since the extension step on $\sigma$ is wasteful (and,
in particular, not a saving step of type (iii)) and since all
$j$-subfaces of
$\sigma$ are light for $j<i$, every $j$-face in $\sigma \cup y$ is light.
Moreover, every $i$-subface of $\sigma \cup y$ that contains $y$ is light as well.

We claim that $\sigma'=\sigma \setminus \{u\} \cup \{y\}$
where the vertex $u$ of $\sigma$ maximizes $r^{\sigma}_i(v)$.
(Since $|H_i^{\sigma}|>0$, there are vertices $v$
in $\sigma$ for which $r^{\sigma}_i(v) > 0$).

Notice that $\sigma$ has ${d-1 \choose i}-r^{\sigma}_i(u)$
more light $i$-subfaces than does
$\tau^u:=\sigma\setminus u$. Namely,
$|L_i^{\tau^u}|= |L_i^{\sigma}|-{d-1 \choose i}+r^{\sigma}_i(u)$.

Combining the fact that every $i$-subface of $\sigma \cup y$ that contains $y$ is light
we see that in $\tau^u_y:=\tau^u\cup y$,
$|L_i^{\tau^u_y}|=|L_i^{\tau^u}|+{d-1\choose i}=|L_i^{\sigma}|+r^{\sigma}_i(u)$.
But since $r^{\sigma}_i(u)>0$, $\tau_y^u$ precedes $\sigma$. In this case $\tau^u_y$ must be a new face
that does not belong to the previous complex, or else it would have been preferred over $\sigma$.
Being a new face, it is
necessarily unsaturated. Since $u$ maximizes $r^{\sigma}_i(u)$ over all vertices in $\sigma$, it follows
that $\tau^u_y$ precedes all other faces created in the expansion. Furthermore, no other face precedes
$\sigma$ or else it would be chosen rather than $\sigma$. Thus $\sigma'=\tau^{y}_{u}$, as claimed.
Notice  that $y\in V_i^{\sigma'}$ and also
$V_i^{\sigma}\subseteq V_i^{\sigma'}$ (Note that every $\emph{i}$-dimensional subfaces of $\sigma'$ that is not contained in
$\sigma$ is light since it contains $y$). Thus $|V_i^{\sigma'}|\ge |V_i^{\sigma}|+1$.
\end{proof}

Consider a chosen non-primary face $\sigma$ and let $i$ be the smallest dimension for which $|H_i^{\sigma}|>0$.
The previous claim implies that
after at most $d$ consecutive wasteful steps the chosen face, $\theta$ precedes $\sigma$ and $|V_i^{\theta}|=d$.
Then $|H_j^{\theta}|=0$ for all $j \le i$ (in particular $|H_i^{\theta}|=0$). By repeating this argument $d-1$ times
we conclude that following every series
of $d(d-1)$ consecutive wasteful steps, a primary face must be chosen:
After at most $d$ consecutive wasteful steps
the chosen face can have no heavy vertices. At the end of the next $d$ consecutive wasteful steps,
the chosen face has no heavy vertices nor heavy edges.
Repeating this argument $(d-1)$ times necessarily leads us to a chosen primary face.

\begin {proposition}
 \label{prop2}
Only saving steps can decrease the number of unsaturated primary faces.
\end{proposition}
\begin{proof} Let $\sigma=a_1,a_2\dots ,a_d$ be a primary face and let $y$ be the vertex that expands it.
Denote the $(d-1)$-face $\{a_1,a_2\dots,a_{i-1},a_{i+1},\dots ,a_d\}\cup \{y\}$ by $\sigma^i$.
Since this is not a saving step of type (i),
$y$ is light. It is also not of type (iii) and so $|H_k^{\sigma^i}|=0$ for every $i=1,\ldots,d$ and $k=1,\ldots, d-2$,	 
so that faces $\sigma^i$ are primary. However this is not a type (ii) saving step, so
 all the $(d-1)$-faces $\sigma^i$ must be new.
Thus the number of unsaturated primary faces has increased by at least $d-1$.
\end{proof}

The proof of Claim~\ref{d_cube} is now complete, since at each step at most $d+1$ faces get covered.
\end{proof}

We can turn now to bound the number of minimal core $m$-simplex complexes $C_m$.
As mentioned, we first specify the heavy $(d-2)$-faces of $C_m$
by specifying a set of $b^{d-1} n^{d-1}$ out of the total of ${n \choose d-1}$
$(d-2)$-faces. Then we select the first
simplex $C_1$ and mark all its $(d-1)$-faces as unsaturated. In order to choose the $i$-th
step we first decide whether it is a saving or wasteful step, and if it is a saving
step, what type it has. There is a total of $d+1$ possible kinds of extensions of the current
$(d-1)$-face: A saving step of type (i), (ii),
or one of the $d-2$ choices of type (iii) (according to dimension), or a wasteful step.
In a saving step the expanding vertex can be chosen in at most $d^{d}bn$ ways.
The number of possible extension clearly never exceeds $n$ and it is this trivial upper bound that we use
for wasteful steps.
Finally we update the labels on the $(d-1)$-faces
of a new simplex. We need to decide which of the unsaturated $(d-1)$-faces
that are already covered by at least two simplices change their status to saturated.
There are at most $2^{d+1}$ possibilities of such an update. As we saw, at least $\frac{m}{d^3}$ of the steps
in such process are saving steps. Consequently we get the following upper bound on $C_d(n,m)$, the
number of minimal core $n$-vertex $d$-dimensional complexes with $m$ simplices.
(In reading the expression below, note that the terms therein correspond in a one-to-one manner to the ingredients
that were just listed).

$$
{{n\choose d-1}\choose {b^{d-1}n^{d-1}}} n^{d+1} (d+1)^{m-1} n^{m-1}   (d^db)^{\frac{m}{d^3}} (2^{d+1})^m
$$
$$\le {n^{d-1}\choose {(d^{2}\delta)^{\frac{d-1}{d}}n^{d-1}}}n^d n^m ((d+1)2^{d+1}d^{\frac{1}{d^2}}b^{\frac{1}{d^3}})^m\le
$$
$$
\le{n^{d-1}\choose (d^{2}m)^{\frac{d-1}{d}}}n^d n^m (2^{d+1}d^2(d^2\delta)^{\frac{1}{d^4}})^m
\le{n^{d-1}\choose (d^{2}m)^{\frac{d-1}{d}}}n^d n^m (2^{d+1}d^3\delta^{\frac{1}{d^4}})^m
$$
\end{proof}

{\bf Proof of Theorem \ref{subcmplx}:}
We show the assertion
with $\delta=\delta(c)=(2^{d+2}d^3c)^{-d^{4}}$. Indeed, consider a complex drawn from $Y_{d}(n,p)$.
Let $X_m=X_m(n,p)$ count the number
of minimal core subcomplexes with $m$ simplices and which are not copies
of $\partial\Delta_{d+1}$. Our argument splits according to whether $m$ is small or large, the dividing line
being $m=m_1=(d^{3} \log n)^{d}$. The theorem speaks only about the range
$m \le m_2=\delta(c)n^d$.
By Lemma~\ref{l:1},

\begin{align*}
&\sum_{m=m_1}^{m_2}\E X_m\le \sum_{m=m_1}^{m_2}C_{d}(n,m)p^m
\le \sum_{m=m_1}^{m_2} (n^{d-1})^{(d^{2}m)^{\frac{d-1}{d}}} n^d\big(2^{d+1}d^3\delta^{\frac{1}{d^{4}}} c\big)^{m} \\
& \le n^{d} \sum_{m=m_1}^{m_2} \left(n^{d-1}\right)^{(d^{2}m)^{\frac{d-1}{d}}}2^{-m}\le n^{d} \sum_{m=m_1}^{m_2} \left(n^{d-1}2^{-\frac{m^{\frac{1}{d}}}{d^2}}\right)^{(d^{2}m)^{\frac{d-1}{d}}}\\
& \le n^{d} \sum_{m=m_1}^{m_2} \left(\frac{1}{n}\right)^{(d^{2}m)^{\frac{d-1}{d}}}=o(1)\\
\end{align*}

It follows that with almost certainty no cores with $m$ simplices occur where $m_2 \ge m \ge m_1$.
We next consider the range $d+3 \le m \le m_1$. Note that a minimal core complex
with $m\ge d+3$ simplices has at most $\frac{d+3}{d+4}m$ vertices.
Let $\Delta(u)$ denote the number of simplices that contain the vertex $u$.
Clearly, if $\Delta(u)>0$ then $\Delta(u)\ge d+1$ (Consider a simplex $\sigma$ that contains $u$.
Every face of the form $\sigma \setminus w$ with $u \neq w \in \sigma$ is covered by a simplex other than $\sigma$).
It is not hard to verify that if some simplex $\sigma$ contains two distinct vertices with $\Delta(u)=\Delta(w)=d+1$
then the complex contains $\partial\Delta_{d+1}$ contrary to the minimality assumption.
Let $t$ be the number of vertices $u$ with $\Delta(u)=d+1$. No simplex
contains two such vertices, so that $t\le \frac{m}{d+1}$. Counting vertices in the complex
according to the value of $\Delta$ we get

$$(d+1)t +(d+2)(v-t)\le (d+1)m$$
where $v$ is the total number of vertices. The conclusion follows.

The expected number of minimal core subcomplexes of $Y_{d}(n,p)$ that contain
$d+3\le m\le m_1$ simplices satisfies
\begin{align*}
\sum_{m=d+3}^{m_1}\E X_m
&\le \sum_{m=d+3}^{m_1}{n\choose \frac{d+3}{d+4}m}
\Big[\Big(\frac{d+3}{d+4}m\Big)^{d+1} p\Big]^m\le \sum_{m=d+3}^{m_1}n^{ \frac{d+3}{d+4}m}
\Big[m^{d+1} p\Big]^m\\
&\le \sum_{m=d+3}^{m_1}\Big(c^{d+4}\frac{\log^{2d(d+1)(d+4)}n}{n}\Big)^{\frac{m}{d+4}} =o(1).
\end{align*}
Consequently, a.a.s. $Y_{d}(n,p)$ contains no minimal core subcomplexes
 of $m$ simplices with $d+3\le m\le (2^{d+2}d^3c)^{-d^{4}}n^d$.
\enp

\section{The Threshold for $d$-Collapsibility}
\label{s:thold}

For a complex $Y \subset \dn^{(d)}$ and a fixed $\tau \in \dn(d-1)$, define a sequence of complexes
$\{\cs_i(Y)\}_{i \geq 0}$
as follows. $\cs_0(Y)=\tau$ and for $i \geq 1$ let $\cs_i(Y)$ be the union of $\cs_{i-1}(Y)$ and the complex generated by  all  the $d$-simplices of $Y$ that contain some $\eta \in \cs_{i-1}(Y)(d-1)$.
Let $\td$ denote the family of all $d$-trees. Consider the events $A_k, D \subset Y_d(n,p)$ given by
$$A_k=\{\cs_k(Y) \in \td\}$$
and
$$D=\{\deg_Y(\eta) \leq \log n {\rm ~~for~all~~}\eta \in \dn(d-1) \}.$$
\begin{claim}
\label{subtree}
Let $k$ and $c>0$ be fixed and $p=\frac{c}{n}$. Then
$$\prob [A_{k+1} \cap D]=1-o(1).$$
\end{claim}
\noindent
{\bf Proof:} Fix $\eta \in \dn(d-1)$. The random variable $\deg_Y(\eta)$ has a binomial distribution
$B(n-d,\frac{c}{n})$, hence by the large deviations estimate
$$\prob [\deg_Y(\eta)> \log n] < n^{-\Omega(\log\log n)}.$$
Therefore $\prob [D]=1-o(1)$. If $Y \in D$ then
$f_0(\cs_{k+1}(Y))=O(\log^{k+1} n)$ and $f_{d-1}(\cs_k(Y))=O(\log^k n)$.
Note that $\cs_{k+1}(Y)$ is a $d$-tree iff in its generation process, we never add a simplex of the form $\eta v$ such that both $\eta \in \dn(d-1)$ and $v\in \dn(0)$ already exist in the complex. Since the number of such pairs is at most
$f_0(\cs_{k+1}(Y))f_{d-1}(\cs_k(Y))$ it follows that
$$\prob[A_{k+1} \cap D] \geq (1-\frac{c}{n})^{f_0(\cs_{k+1}(Y))f_{d-1}(\cs_k(Y))} \geq $$
$$
(1-\frac{c}{n})^{O(\log^{2k+1} n)} =1-o(1)~~.$$
{\enp}
For $Y \subset \dn^{(d)}$ let $r(Y)=f_d(R_{\infty}(Y))$  be the number of $d$-simplices remaining in $Y$
after performing all possible $d$-collapsing steps.
For $\tau \in \dn(d-1)$ let $\Gamma(\tau)=\{\sigma \in \dn(d): \sigma \supset \tau\}$.
\begin{claim}
\label{exps}
Let $0<c<\gamma_d$ be fixed and $p=\frac{c}{n}$. Then for any fixed  $\tau \in \dn(d)$:
\begin{equation}
\label{sinr}
\prob[R_{\infty}(Y) \cap \Gamma(\tau) \neq \emptyset]=o(1)~.
\end{equation}
\end{claim}
\noindent
{\bf Proof:} Let $\delta>0$. Since $c<\gamma_d$
$$
\lim_{k \rightarrow \infty} \rho_d(k,c)=\rho_d(c)=1.
$$
Choose a fixed $k$ such that $$\rho_d(k,c) >1-\frac{\delta}{3}.$$
Claim \ref{subtree} implies that if $n$ is sufficiently large then
$$\prob[A_{k+1} \cap D] \geq 1-\frac{\delta}{3}.$$
Next note that if $Y \in A_{k+1}$ then $\cs_{k+1}=\cs_{k+1}(Y)$ can be generated by the following inductively defined random process:
$\cs_0=\tau$. Let $0 \leq i \leq k$. First generate $T=\cs_i$ and let
$\cu$ denote all $\tau' \in T(d-1)$  such that $dist_T(\tau,\tau')=i$. Then,
according to (say) the lexicographic order on $\cu$, for each $\tau' \in \cu$
pick $J$ new vertices $z_1,\ldots,z_J$ according to the  binomial distribution $B(n-n',\frac{c}{n})$,
where $n'$ is the number of vertices that appeared up to that point, and add
the $d$-simplices $z_1\tau', \ldots,z_J\tau'$ to $T$.
Note that the process described above is identical to the $d$-tree process of Section \ref{s:tree},
except for the use of the binomial distribution $B(n-n',\frac{c}{n})$ instead of the Poisson distribution $Po(c)$. Now if $Y \in A_{k+1} \cap D$ then $n'=O(\log^{k+1}n)$ at all stages of this process.
It follows that if $n$ is sufficiently large then the total variation distance between the distributions $\cs_{k+1}(Y)$ and $\td(k+1,c)$ is less then $\frac{\delta}{3}$.
Denote by ${\cal C}'_d(k+1,c)$ the event that $\cs_{k+1}(Y)$ is in  $A_{k+1}$ and collapses to $\tau$ in at most $k$
pruning steps.
The crucial observation now is that if $Y \in {\cal C}'_d(k+1,c)$
then $R_{\infty}(Y) \cap \Gamma(\tau) = \emptyset$. It follows that
$$
\prob[R_{\infty}(Y) \cap \Gamma(\tau) \neq \emptyset] \leq $$ $$(1-\prob[A_{k+1} \cap D]) +
\prob[Y \not\in {\cal C}'_d(k+1,c)] \leq
$$
$$
(1-\prob[A_{k+1} \cap D]) +d_{TV}(\cs_{k+1}(Y),\td(k+1,c))+
(1-\prob[{\cal C}_d(k+1,c)]) \leq $$
$$
\frac{\delta}{3}+\frac{\delta}{3}+\frac{\delta}{3}=\delta.$$
{\enp}
Let $$\cg(Y)=\{\tau \in \dn(d-1): \rinf \cap \Gamma(\tau) \neq \emptyset\}$$
and let $g(Y)=|\cg(Y)|$. For a family $\cg \subset \dn(d-1)$ let
$w(\cg)$ denote the set of all $d$-simplices $\sigma \in \dn(d)$ all of whose $(d-1)$-faces are contained in $\cg$. Using Claim \ref{exps} we establish the following
\begin{theorem}
\label{rrinf}
Let $\delta>0$ and $0<c<\gamma_d$ be fixed and let $p=\frac{c}{n}$.
Then
$$\prob[f_d(\rinf)>\delta n^d]=o(1).$$
\end{theorem}
\noindent
{\bf Proof:} Let $0<\epsilon=\epsilon(d,c,\delta)<1$ be a constant whose value
will be fixed later. Clearly
$$\prob[f_d(\rinf)>\delta n^d] \leq $$ $$\prob[g(Y) >\epsilon \delta n^d] +\prob[g(Y) \leq
\epsilon \delta n^d ~~~\&~~~ f_d(\rinf)> \delta n^d].
$$
To bound the first summand note that $E[g]=o(n^d)$ by Claim \ref{exps}.
Hence by Markov's inequality
$$
\prob[g(Y)>\epsilon \delta n^d] \leq (\epsilon \delta n^d)^{-1} E[g] =o(1).$$
Next note that
$$
\prob[g(Y) \leq
\epsilon \delta n^d ~~~\&~~~ f_d(\rinf)> \delta n^d] \leq $$
$$\sum_{\{\cg \subset \dn(d-1):|\cg|=\epsilon \delta n^d\}} \prob [|w(\cg)\cap Y(d)| >\delta n^d].$$
Fix a
$\cg \subset \dn(d-1)$ such that $|\cg|=\epsilon \delta n^d$.
By the Kruskal-Katona theorem there exists a $C_1=C_1(d,\delta)$ such that
$$
N=|w(\cg)| \leq C_1 \epsilon^{\frac{d+1}{d}}n^{d+1}.$$
Applying the large deviation estimate for the binomial distribution $B(N,\frac{c}{n})$
and writing $C_2=\frac{e c C_1}{\delta}$ we obtain
$$
\prob[|w(\cg)\cap Y(d)|>\delta n^d] \leq (C_2 \epsilon^{\frac{d+1}{d}})^{\delta n^d}.$$
On the other hand
$$
\binom{\binom{n}{d}}{\epsilon\delta n^d} \leq (\frac{e}{\epsilon \delta})^{\epsilon \delta n^d}.$$
Choosing $\epsilon$ such that $$(\frac{e}{\epsilon \delta})^{\epsilon} C_2 \epsilon^{\frac{d+1}{d}}<e^{-1}$$
it follows that
$$
\prob[g(Y) \leq
\epsilon \delta n^d ~~~\&~~~ f_d(\rinf)> \delta n^d] \leq \exp(-\delta n^d).$$
{\enp}
\noindent
{\bf Proof of Theorem \ref{collapse}:} Let $c < \gamma_d$ and $p=\frac{c}{n}$.
By Theorem \ref{subcmplx} there exists a $\delta>0$ such that
a.a.s. any non-$d$-collapsible subcomplex $K$ of $Y \in Y_d(n,\frac{c}{n})$ such that
$f_d(K) \leq \delta n^d$ contains the boundary of a $(d+1)$-simplex.
It follows that
$$\prob ~[Y~{\rm non~} d-{\rm collapsible}~|~Y \in  \cf_{n,d}]=\prob ~[f_d(\rinf)>0~|~Y \in  \cf_{n,d}] \leq$$
$$\prob ~[f_d(\rinf)>\delta n^d]\cdot \prob~[~Y \in  \cf_{n,d}]^{-1}+\prob ~[0<f_d(\rinf) \leq \delta n^d~|~Y \in  \cf_{n,d}].$$
The first summand is $o(1)$ by Theorem \ref{rrinf}, and the second summand is $o(1)$ by Theorem \ref{subcmplx}.
{\enp}

\section{Concluding Remarks}
\label{s:remark}

Let us remark that, estimating the vanishing probabilities by a function
which tends to 0 as $n\to\infty$ faster that $n^{-d}$
one may show a random process statement
slightly stronger than Theorem~\ref{subcmplx}
(see  \cite{L}, where a similar result is shown
for the $k$-core of
random graphs). More specifically, let us define
the $d$-dimesional random process
$\mathcal{Y}_d=\{Y_d(n,M)\}_{M=0}^{\binom n{d+1}}$ as the Markov
chain whose stages are simplicial complexes, which starts with
the full $(d-1)$-dimensional skeleton of $\Delta_{n-1}$ and no
$d$-simplices, and in each stage $Y_d(n,M+1)$ is obtained from
$Y_d(n,M)$ by adding to it one $d$-simplex chosen uniformly  at random
from all the $d$-simplices which do not belong to $Y_d(n,M)$.
Then, the following  holds.

\begin{theorem}
\label{subcmplxprocess}
There exists a constant $\alpha=\alpha(d)>0$ such that
for almost every  $d$-dimesional random process
$\mathcal{Y}_d=\{Y_d(n,M)\}_{M=0}^{\binom n{d+1}}$
  there exists a stage $\hat M=\hat M(\mathcal{Y}_d)$
such that the core of $Y_d(n,\hat M)$ is of the size
$O(1)$ and consists of boundaries of $(d+1)$-simplices,
while the core of $\hat M=\hat M(\mathcal{Y}_d+1)$
contains at least $\alpha n^{d}$ $d$-simplices.
\end{theorem}

Many questions remain open. The most obvious ones are

\begin{itemize}
\item
What is the threshold for $d$-collapsibility of random simplicial complexes in $\cf_{n,d}$ ? We conjecture that it is indeed
$p=\gamma_d/n$.
\item
Derive better upper bounds on the threshold for the nonvanishing of $H_d(Y)$. We intend to return
to this subject in subsequent papers.
\item
In particular does the threshold for the vanishing of $H_d(Y)$ depend on the underlying field?
\item
Although this question is implicitly included in the above two questions, it is of substantial interest
in its own right: Can you show that the two thresholds (for $d$-collapsibility and for the
vanishing of the top homology) are
distinct? We have good reasons to think that the two thresholds are, in fact, quite different.
In particular, although $d$-collapsibility is a sufficient condition for the vanishing of $H_d$, there is only a vanishingly small probability that a
random simplicial complex with trivial top homology is $d$-collapsible.
\end{itemize}

\end{document}